\title{On the energy equality for solutions to Newtonian and non-Newtonian fluids}
\author{Hugo Beir\~{a}o da Veiga$^{1,}$ \footnote{Partially supported  by FCT (Portugal) under grant UID/MAT/04561/3013.}\qquad Jiaqi Yang$^{2,}$\footnote{Hugo Beir\~{a}o da Veiga (\texttt{bveiga@dma.unipi.it}) and Jiaqi Yang (\texttt{yjq@imech.ac.cn})}}
\date{
\small $^1$ Department of Mathematics, Pisa University, Pisa, Italy\\
\small $^2$ Key Laboratory for Mechanics in Fluid Solid Coupling Systems, Institute of Mechanics, Chinese Academy of Sciences, Beijing 100190, China}
\documentclass[11pt]{article}
\usepackage{amsfonts}
\usepackage{mathrsfs}
\usepackage{amsmath}
\usepackage{amssymb,extarrows}
\usepackage{multicol}
\usepackage{float}
\usepackage{makeidx}
\usepackage{layout}
\usepackage{array}
\usepackage{a4wide}
\usepackage{boxedminipage}
\usepackage{hyperref}
\usepackage{latexsym}
\usepackage{color}
\usepackage{dsfont}
\usepackage{graphicx}
\usepackage{ulem}
\usepackage{amsthm}

\newtheorem{theorem}{Theorem}[section]

\theoremstyle{remark}
\newtheorem{remark}{Remark}[section]

\theoremstyle{definition}

\numberwithin{equation}{section}

\newcommand{\p}{\partial}
\newcommand{\e}{\epsilon}

\newcommand{\R}{\mathbb{R}}

\newcommand{\f}{\frac}
\newcommand{\n}{\nabla}
\newcommand{\tha}{\theta}

\newcommand{\ed}{\end{document}}
\newcommand{\na}{{\nabla}}
\newcommand{\pa}{{\partial}}

\newcommand{\Om}{{\Omega}}

\begin{document}
\maketitle
\begin{abstract}
We are concerned with the energy equality for weak solutions to Newtonian and non-Newtonian incompressible fluids. In particular, the results obtained for non-Newtonian fluids, after restriction to the Newtonian case, equal or improve the known results. Furthermore, a new result below allows an interpretation of the classical Shinbrot condition, coherent to the more recent results in the literature.
%\vskip 0.1 true cm
%\noindent\textbf{Keywords}: Navier-Stokes equations; Slip boundary conditions; Two components; Non-flat boundaries
%\vskip 0.1 true cm
%\noindent\textbf{2010 Mathematical Subject Classification}: 35Q30, 35B65
%\vskip 0.1 true cm
\end{abstract}


\vspace{0.2cm}

\vspace{0.2cm}
\section{Introduction. The Problem.}
We are interested on the energy equality for weak solutions to Newtonian ($r=2$) and non-Newtonian ($r\neq2$) incompressible fluids
\begin{equation}\label{eq:non-Newtonian system}
\begin{cases}
u_t+u\cdot\n u-\text{div}\left(|D(u)|^{r-2}D(u)\right)+\n p=0, &\text{in $\Omega\times(0,T)$},\\
\n\cdot\, u=0, &\text{in $\Omega\times(0,T)$},\\
u=0,&\text{in $\p\Omega\times(0,T)$},\\
u(\cdot,0)=u_0, &\text{in $\Omega$},
\end{cases}
\end{equation}
where
\[
D(u)=\f{1}{2}\left(\n u+(\n u)^T\right)
\]
is the stress tensor, and $\Omega\subset\R^3$ is a bounded domain, with smooth boundary $\p\Omega$.\par%

\vspace{0.2cm}

Our main results are the Theorems \ref{thm:BC} and \ref{thm:BV} below. However, before stating these two results, we present an overview on some related, previous results.  This is the aim of the next two sections. In particular, in section \ref{The Berselli-Chiodaroli Results.}, we recall Berselli and Chiodaroli's Theorem \ref{thm:BC}, in reference
\cite{BC}. This reference gave us motivation, and lines of reasoning, leading to the present extension. See also the
fourth chapter of G.P. Galdi's work \cite{galdi-2000}. This last reference also furnishes an outstanding treatment of the more important classical topics on the Navier-Stokes equations.\par%
It seems suitable to inform the interested readers that, as far as we know from the authors, reference \cite{BC} is not the final version.
\section{On some Newtonian Classical Results. The Shinbrot SH-measure.}
In this section we consider the Newtonian case, $r=2$, which can be written in the equivalent, more usual form,
\begin{equation}\label{eq:NS}
\begin{cases}
\pa_t u+\,u\cdot\nabla\,u-\Delta u+\nabla p=0\,,\quad &\textrm{in $\Omega\times(0,T)$}\,,\\
\nabla \cdot\,u=0\,,\quad& \textrm{in $\Omega\times(0,T)$}\,,\\
u=0\,,\quad& \textrm{on $\pa\Omega\times(0,T)$}\,,\\
u(\cdot,\,0)=\,u_0 \quad &\textrm{in $\Omega$}\,,
\end{cases}
\end{equation}

\vspace{0.2cm}

In reference \cite{shinbrot} M.Shinbrot shows that if a weak solution $\,u\,$ to the Navier-Stokes equations \eqref{eq:NS} satisfies
\begin{equation}\label{shin}
u \in\,L^p (0,\,T; L^q(\Omega))\,,
\end{equation}
where
$$
\frac 2p + \frac 2q =\,1\,,\quad  \textrm{and} \quad q\geq\,4,
$$
then it satisfies the energy equality. This result is a generalization of previous results due to G.Prodi \cite{Prodi}  and J.L.Lions \cite{lions}, where these authors proved the above result for $p=\,q=\,4\,.$%

\vspace{0.2cm}
Shinbrot condition, denoted in the sequel by SH condition, may furnish a possible measure of the strength of sufficient conditions, of integral type, to guarantee the energy equality. In a rough sense, it is a (more questionable) counterpart of the well known Prodi-Serrin's condition in the regularity theory for solutions to the Navier-Stokes equations. Note that Shinbrot's result does not depend on the space dimension. This is quite natural, if one takes into account that the main rule to prove the result is played by the time and not by the space.\par%
If the SH measure \eqref{shonas} below is significant as a measure of strength, we expect that results obtained under assumption \eqref{shin}, by essentially appealing to the same proof (even if for a set of distinct values of the parameters $p$ and $q$), must enjoy the same value for the quantity
\begin{equation}\label{shonas}
\theta(p,\,q) \equiv\,\frac 2p + \frac 2q\,,
\end{equation}
independently of the particular values of $p$ and $q$. Exactly as happens to the result proved by Shinbrot.%

\vspace{0.2cm}

After the above 1974 Shinbrot's paper \cite{shinbrot} new results, enjoying a better SH measure, appeared, see \cite{BC} for references.  Application of the SH condition to these stronger results have brought to light a negative feature of the SH condition. It appears that it strongly depends on the particular values of the parameters $p,\,q\,.$ In particular the SH condition presents a peak for some particular values of the parameters. This would show that SH does not give a correct measure of the strength. This happens, in particular, to the quite strong results obtained by L.C. Berselli and E. Chiodaroli in reference \cite{BC}. One of our aim is to solve, positively, this anomaly. This study leads to the Theorem \ref{thm:BV} below.\par%
 The other main aim is to extend the Berselli-Chiodarolli results to non-Newtonian fluids. This leads to the Theorem \ref{thm:BC} below. Actually, in this theorem, we also improve some of their results.\par%
Before going inside our main results, we recall the Berselli-Chiodarolli results.
\section{The Berselli-Chiodaroli Results.}\label{The Berselli-Chiodaroli Results.}
In the recent paper \cite{BC} L.C. Berselli and E. Chiodaroli gave a very important contribution to
the study of the energy equality to solutions of the Navier-Stokes equations
by improving, in a quite substantial way, several previous known results. They essentially proved
new sufficient conditions for the energy equality involving integrability conditions on the gradient of the velocity.
Note that the SH condition allows comparison between strengths of sufficient conditions depending on integrability assumptions on the velocity, with strengths of similar conditions on the gradient of the velocity, by appealing to well known Sobolev's embedding theorems, to transform integrability exponents of $\na u$ into "equivalent" integrability exponents for $u$.

\vspace{0.2cm}

The main result in BC is the following. For covenience we replace the (i)-notation used in this last reference by (j)-notation:
\begin{theorem}\label{bers-chio}
Let $u_0 \in H $ be a Leray-Hopf weak solution of the Navier-Stokes equations \eqref{eq:NS}, and assume that $\,\na u \in\, L^p(0,\,T; L^q(\Om))\,$ for the following ranges of the exponents $p,\,q$:
\begin{equation}\label{pqcases}
\begin{cases}
(j)\quad \frac32 <q< \frac95 \,, \quad \textrm{and}\quad \na u \in\, L^{\frac{q}{2q-\,3}}(0,\,T; L^q(\Om))\,\,;\\
(jj)\quad \frac95 \leq q<3 \,, \quad \textrm{and}\quad \na u \in\, L^{\frac{5q}{5q-\,6}}(0,\,T; L^q(\Om))\,\,;\\
(jjj)\quad 3\leq q \,, \quad \textrm{and}\quad \na u \in\, L^{1+\,\frac{2}{q}}(0,\,T; L^q(\Om))\,\,.\\
\end{cases}
\end{equation}
Then $u$ satisfies the energy equality
\begin{equation}\label{eneq}
\| u(t)\|_2^2 \,+\, 2\,\int_{0}^{t} \,\|\n u(\tau)\|_2^2 \,d\tau=\,\| u_0\|_2^2\,.
\end{equation}
Hereafter, we define $H$ by the completion of $\mathcal{V}$ in $L^2(\Omega)$, where $\mathcal{V}$ is the space of smooth divergence free vectors with compact support in $\Omega$.
\end{theorem}
In reference \cite{BC} other main results are proved. In particular, in Section 1.2, "Energy conservation and the Onsager conjecture", the authors proved that if
$$
u \in  L^{\frac{13}{9}}(0,\,T; C^{0,\,\frac13}(\Om))
$$
then $u$ verifies the energy equality. This result substantially improves the previous main results known in the literature.\par%
To end this section we note that, at least in some cases, belonging to the Leray-Hopf class may be simply replaced by belonging to the space $\,L^2_{loc}(\R^3 \times (0,\,T)\,)\,,$ in the distributional sense. See \cite{galdi-dist}.\par%
\section{On the Non-Newtonian Case.}
In reference \cite{Yang} J.Yang succeed in extending Shinbrot's result to solutions of the non-Newtonian system \eqref{eq:non-Newtonian system}. On the other hand, as still referred, in reference \cite{BC} the authors improved Shinbrot's results. Hence, by putting together both results, one immediately considers the possibility of extending the results proved in reference \cite{BC} to the non-Newtonian fluids. This aim is reached in Theorem \ref{thm:BC} below where, in particular, the result concerning the case $\,q>3\,$ is improved.\par%
To guarantee the existence of weak solutions to system \eqref{eq:non-Newtonian system} we assume everywhere that $\,r>\,\f85\,.$ See \cite{Yang}, Remark 1 (in fact this restriction is included, case by case, in all the distinct restrictions imposed in the sequel).%
\begin{remark}\label{rem:rego}
It is worth noting that for $\,r\geq\,\frac{11}{5}\,$ weak solutions are strong, see \cite{Ladyzhenskaya,malek,malek1}. This leads to the energy equality. See also the Remark 2 in \cite{Yang}, and comments. However, by taking advantage of the treatment of the other cases considered below, a direct proof of the above claim follows without too much additional material. So, for completeness, and for the reader's convenience, we do not restrict our proofs to the case $\,r<\,\frac{11}{5}\,$.
\end{remark}
\section{Main Results}
We start from the extension of Berselli and Chiodarolli's results to the non-Newtonian case.
\begin{theorem}\label{thm:BC}
Let $u_0\in H$ and let $u$ be a Leray-Hopf weak solution of \eqref{eq:non-Newtonian system} in a smooth bounded domain $\Omega$. Let us assume that $\n u\in L^p(0,T;L^q(\Omega))$ for the following ranges of the exponents $p\,,q$:
\begin{description}
  \item[(i)$_1$] $\f{9}{5}< r\leq2$, $\f{9-3r}{2}<q\leq\f95$ and $p=\f{q(5r-9)}{3r+2q-9}$, i.e. $\n u\in L^{\f{q(5r-9)}{3r+2q-9}}(0,T;L^q(\Omega))$;
  \item[(i)$_2$] $2<r<3$, $\f{3r}{5r-6}\leq q\leq\f95$ and $p=\f{q(5r-9)}{3r+2q-9}$, i.e. $\n u\in L^{\f{q(5r-9)}{3r+2q-9}}(0,T;L^q(\Omega))$;
  \item[(i)$_3$] $r\geq3$, $1< q\leq\f95$ and $p=\f{q(5r-9)}{3r+2q-9}$, i.e. $\n u\in L^{\f{q(5r-9)}{3r+2q-9}}(0,T;L^q(\Omega))$;
  \item[(ii)] $q>\f95\,$ and $p=\f{5q}{5q-6}$, i.e. $\n u\in L^{\f{5q}{5q-6}}(0,T;L^q(\Omega))$\,.
\end{description}
Then $u$ satisfies the energy equality
\begin{equation*}
\|u(t)\|_2^2+2\int_0^t\|D(u)(\tau)\|^r_rd\tau=\|u_0\|_2^2\,.
\end{equation*}
\end{theorem}
\begin{remark}\label{rem:MAIN}
It is worth noting that the results stated in the four i-items in Theorem \ref{thm:BC} glue each other on the values reached by the parameters $r$ and $q$ at the extreme points of their intervals of definition (obvious details concerning the positions of the symbols $<,\,\leq,\,>,\,$ and $\,\geq\,,$ are left to the reader).  For $r=2$ the results in (i)$_1$ and  (i)$_2$ glue perfectly and, similarly, the results in  (i)$_2\,$ and (i)$_3\,$ glue for the value $r=3\,.$ Further, for $\,q=\f95\,,$ the result in item (ii) glue to the results in the three (i) items. Note that, for  $\,q=\f95\,,$
$$
\f{q(5r-9)}{3r+2q-9}=\,=\f{5q}{5q-6}\,,
$$
independently of $r$. Concerning $r$, recall also the above remark \ref{rem:rego}.
\end{remark}
\begin{remark}\label{rem:G-N}
By appealing in our proofs to Gagliardo-Nirenberg's inequalities we unify the treatment of all cases $q>\,\f95 \,$, and also improve the result obtained in reference \cite{BC}, in the case $q>\,3$. In particular, the apparently singular case $\,q=\,3\,,$ is treated inside the generic proof. For $r=2$, our case (i)$_1$ coincides with case (j) in \cite{BC} theorem \ref{bers-chio}, and our case (ii) coincides with case (jj) for $\f95<\,q<\,3\,,$ and improves case (jjj) for $q>3$ we since
$$
\f{5 q}{5q-\,6}<\,1+\,\f{2}{q}\,.
$$
\end{remark}
\begin{remark}
In the above theorem the ``best exponent'' is $q=\f95$, at the light of the corresponding SH measure, equal to $\f{10}{9}\,.$ Moreover, $p=3$ is independent of $r$. Actually, when $q\geq\f95$, the exponent $p$ is independent of $r$.  This is reasonable since we only use the estimate $u\in L^{\infty}(0,T;H)$ in the proof of the case $q\geq\f95$, and this estimate is independent of $r$. On the contrary, in case (i) we appeal to the $r-$ weak estimate for weak solutions, which shows that $\,\|\n u\|_{r,r}\,$ is bounded.
\end{remark}

\vspace{0.2cm}

Next we state our second main result. This result is independent of $r$ since our proof do not use the  $r-$weak estimate.%
\begin{theorem}\label{thm:BV}
Let $u_0\in H$ and let $u$ be a Leray-Hopf weak solution of \eqref{eq:non-Newtonian system} in a smooth bounded domain $\Omega$.
Assume that
\begin{equation}\label{eq:shcond}
u\in L^{p_1}(0,T;L^{q_1})\cap L^{\f{9q}{8q-9}}(0,T;W^{1,q})\,,
\end{equation}
where
\begin{equation}\label{eq:pumqum}
\frac{2}{p_1}+\,\frac{2}{q_1}=\,\frac{10}{9}\,,
\end{equation}
and the exponents $\,q$ and $ q_1\,$ satisfy one of the following assumptions:\\
\begin{description}
 \item[(i)] $\f98<q\leq\f95$ and $q_1 \geq\,\f{2 q}{q-\,1};$
 \item[(ii)] $ q>\f95$ and $\f95<\,q_1\leq\,\f{2q}{q-\,1}\,.$
\end{description}
Then $u$ satisfies the energy equality
\begin{equation*}
\|u(t)\|_2^2+2\int_0^t\|D(u)(\tau)\|^r_rd\tau=\|u_0\|_2^2\,.
\end{equation*}
In particular both spaces appearing in equation \eqref{eq:shcond} have SH measure $\,\f{10}{9}\,.$ Hence the intersection space has the same SH measure $\,\f{10}{9}\,$. So, the SH strength of all the above results coincide.
Further, for $q=\,\f95\,$, the above intersection coincides with the second space $\,L^{\f{9q}{8q-9}}(0,T;W^{1,q})$.
\end{theorem}
Concerning the second-last statement in the theorem, note that $\,W^{1,q} \subset L^{q^*}\,$ with $q^*=\f{3q}{3-q}\,(q<3)$ and that
$$
\f{2}{\f{9q}{8q-9}}+\,\f{2}{q^*}=\,\f{10}{9}\,,
$$
independently of $q$. So, as stated above, both spaces considered in equation \eqref{eq:shcond} have the maximum Shinbrot measure $\,\frac{10}{9}\,.$ Hence the intersection space has the same maximum SH measure, independently of the value of $q$.\par%
A quite significant point is that for $\,q=\,\frac95\,$ one has
$$
L^{\frac{9\, q}{8 q -9}} (0,\,T; W^{1,\,q})\subset\,L^{\frac{18\, q}{9+q}} (0,\,T; L^{\frac{2q}{q- 1}})\,.
$$
So the intersection space in \eqref{eq:shcond} coincides with the space $\, L^{\f{9q}{8q-9}}(0,T;W^{1,q})\,$. In other words, for $\,q=\,\frac95\,$ our condition coincides with the best condition obtained in theorem \ref{bers-chio}, namely
$$
u \in L^3 (0,\,T; W^{1,\,\frac{9}{5}})\,,
$$
implies the energy equality.%

\vspace{0.2cm}

Open problem: To show that if
$$
u \in  L^{\frac{3 r}{ 2 r-\,3}} (0,\,T; W^{\frac13,\,r})\,,
$$
where $\,r \in (\frac32 ,\,9)\,,$ then $\,u\,$ enjoys the energy equality.

\begin{remark}
Note that for $q=\f95$, the results in items (i) and (ii) glue perfectly.
\end{remark}
Clearly, the Remark \ref{rem:G-N} also applies here.
\section{Sketch of the Proofs}
We follow some of the main lines of the proofs shown in reference \cite{BC}, to which we are strongly indebted. See also the related proof of Theorem 4.1, in reference \cite{galdi-2000}.
\begin{proof}[Proof of Theorem \ref{thm:BC}]
Define the typical mollification operator $(\cdot)_{\epsilon}$
\[
(\Phi)_{\epsilon}=\int_0^tk_{\epsilon}(t-\tau)\Phi(\tau)d\tau\,,
\]
where the infinitely differentiable, real-valued, nonnegative, even function $k_{\epsilon}$ , has support in $(-\e,\e)$, and integral equal to one.
Next, let $\{u_m\}$ be a sequence in $C^{\infty}(0,\infty;C^{\infty}_0(\Omega))$ converging to $u$ in
\[
L^r(0,T;V_r)\cap L^{p}(0,T;W^{1,q})\,,
\]
where $V_r$ is denoted by the completion of $\mathcal{V}$ in $W^{1,r}(\Omega)$.
As in \cite{BC}, one has
\begin{equation*}
\begin{split}
(u(t_0),(u_m)_{\e}(t_0))=&(u_0,(u_m)_{\e}(0))\\
&+\int_0^{t_0}\left[\left(u,\f{\p(u_m)_{\e}}{\p t}\right)-(u\cdot\n u,(u_m)_{\e})-\left(|D(u)|^{r-2}D(u),D\left(\left(u_m\right)_{\e}\right)\right)\right]dt\,.
\end{split}
\end{equation*}
The main point is to single out sufficient conditions to guarantee the existence of a suitable limit, as $m\to\infty$. The real obstacle is the non-linear term. We write the nonlinear term as
\begin{equation*}
\begin{split}
\int_0^{t_0}(u\cdot\n u,(u_m)_{\e})dt=&\int_0^{t_0}(u\cdot\n u,(u_m)_{\e}-(u)_{\e})dt\\
&+\int_0^{t_0}(u\cdot\n u,(u)_{\e}-u)dt+\int_0^{t_0}(u\cdot\n u,u)dt\,.
\end{split}
\end{equation*}
We start by proving that
\begin{equation}\label{eq:non=0}
\int_0^{t_0}(u\cdot\n u,u)dt=0\,.
\end{equation}
First, integrating by parts we get
\begin{equation*}
\int_0^{t_0}(u\cdot\n u_m,u_m)dt=0\,.
\end{equation*}
Let us consider
\begin{equation*}
\begin{split}
&\left|\int_0^{t_0}(u\cdot\n u_m,u_m)dt-\int_0^{t_0}(u\cdot\n u,u)dt\right|\\
&\leq \left|\int_0^{t_0}(u\cdot\n u_m,(u_m-u))dt\right|+\left|\int_0^{t_0}(u\cdot\n (u_m-u),u)dt\right|\,.
\end{split}
\end{equation*}
and show the convergence of the two terms on the right hand side to zero in the two different cases, which implies \eqref{eq:non=0}. In the following, we set $q^*=\f{3q}{3-q}$, $q'=\f{q}{q-1}$ and $\tilde{q}=2q'$.

(i) By Gagliardo-Nirenberg inequalities, E. Gagliardo \cite{gagliardo}, L. Nirenberg \cite{nirenberg} (for a more recent reference see for example \cite{Galdi} Lemma II.3.3) we have
\begin{equation*}
\|u\|_{\tilde{q}}\leq \|u\|^{\tha}_{q^*}\|\n u\|^{1-\tha}_r\,,
\end{equation*}
where
\begin{equation*}
\f{1}{\tilde{q}}=(1-\tha)\left(\f1{r}-\f13\right)+\f{\tha}{q^*}\,,
\end{equation*}
which gives $\tha:=\f{5qr-3r-6q}{6(r-q)}$. To guarantee $0\leq \tha\leq1$, we need restrict $\f{3r}{5r-6}\leq q\leq\f95$ . Note that when $r\geq3$, from the second exceptional case of \cite{Galdi} Lemma II.3.3, $\tha$ can not equal to $0$, i.e. $q$ can not equal to $\f{3r}{5r-6}$.
Now, we can estimate, by Sobolev embeddings, the integral as follows
\begin{equation*}
\begin{split}
\left|\int_0^{t_0}(u\cdot\n(u_m-u),u)dt\right|\leq&\left|\int_0^{t_0}\|u\|_{\tilde{q}}\|\n(u_m-u)\|_q\|u\|_{\tilde{q}}dt\right| \\
\leq&\left|\int_0^{t_0}\|u\|^{2\tha}_{q^*}\|\n u\|^{2(1-\tha)}_{r}\|\n(u_m-u)\|_qdt\right|\\
\leq&\left|\int_0^{t_0}\|\n u\|^{2\tha}_{q}\|\n u\|^{2(1-\tha)}_{r}\|\n(u_m-u)\|_qdt\right|\\
\leq&\|\n u\|^{2\tha}_{p,q}\|\n u\|^{2(1-\tha)}_{r,r}\|\n(u_m-u)\|_{p,q}\,,
\end{split}
\end{equation*}
where we have applied H\"{o}lder inequality in the time variable with exponents $x\,,y\,,z$ such that
\begin{equation*}
2(1-\tha)x=r,\quad2\tha y=p,\quad z=p\,, \quad \f{1}{x}+\f{1}{y}+\f{1}{z}=1\,,
\end{equation*}
which gives $p=\f{q(5r-9)}{3r+2q-9}$. To guarantee that the denominator of $p$ has a positive value, we need restrict $q>\f{9-3r}{2}$.
This shows that $\left|\int_0^{t_0}(u\cdot\n (u_m-u),u)dt\right|\to0$.
Similarly, we have
\begin{equation*}
\begin{split}
&\left|\int_0^{t_0}(u\cdot\n u_m,(u_m-u))dt\right|\\
&\leq\left|\int_0^{t_0}\|u\|_{\tilde{q}}\|\n u_m\|_q\|u_m-u\|_{\tilde{q}}dt\right| \\
&\leq\left|\int_0^{t_0}\|u\|^{\tha}_{q^*}\|\n u\|^{(1-\tha)}_{r}\|\n u_m\|_q\|u_m-u\|^{\tha}_{q^*}\|\n(u_m-u)\|^{(1-\tha)}_{r}dt\right|\\
&\leq\left|\int_0^{t_0}\|\n u\|^{\tha}_{q}\|\n u\|^{(1-\tha)}_{r}\|\n u_m\|_q\|\n(u_m-u)\|^{\tha}_{q}\|\n(u_m-u)\|^{(1-\tha)}_{r}dt\right|\\
&\leq\|\n u\|^{\tha}_{p,q}\|\n u\|^{(1-\tha)}_{r,r}\|\n(u_m-u)\|_{p,q}\|\n(u_m-u)\|^{\tha}_{p,q}\|\n(u_m-u)\|^{(1-\tha)}_{r,r}\,.
\end{split}
\end{equation*}
This implies that $\left|\int_0^{t_0}(u\cdot\n (u_m-u),u)dt\right|\to0$.\par%
From the above proof, in the three (i) cases we must restrict the values of $\,q\,$ simultaneously to values $q\geq\f{3r}{5r-6}$ ( $q>\f{3r}{5r-6}$, when $r\geq3$), and $q>\f{9-3r}{2}$. These restrictions hold in all cases since:
$\,\f{3r}{5r-6}\leq\f{9-3r}{2}\,,\, \textrm{when}\, \f95<r\leq2\,;\quad$%
$\,\f{3r}{5r-6}>\f{9-3r}{2}\,, \textrm{when} \, 2<r<3\,.\quad$%
And $\,\f{9-3r}{2}<\f{3r}{5r-6}\leq1 \,, \textrm{when} \, r\geq3\,.$

\vspace{0.2cm}

(ii) In case $q>\f95$ we have $\n u\in L^p(0,T;L^q(\Omega))$ with $p=\f{5q}{5q-6}$.  By Gagliardo-Nirenberg inequalities again, we have
\begin{equation*}
\|u\|_{\tilde{q}}\leq \|u\|^{\tha}_2\|\n u\|^{1-\tha}_q\,,
\end{equation*}
where
\[
\f{1}{\tilde{q}}=(1-\tha)\left(\f1{q}-\f13\right)+\f{\tha}{2}\,.
\]
Hence $\tha=\f{5q-9}{5q-6}$. Now, we can estimate the integral as follows
\begin{equation*}
\begin{split}
\left|\int_0^{t_0}(u\cdot\n(u_m-u),u)dt\right|\leq&\left|\int_0^{t_0}\|u\|_{\tilde{q}}\|\n(u_m-u)\|_q\|u\|_{\tilde{q}}dt\right| \\
\leq&\left|\int_0^{t_0}\|u\|^{2\tha}_{2}\|\n u\|^{2(1-\tha)}_{q}\|\n(u_m-u)\|_qdt\right|\\
\leq&\|u\|^{2\tha}_{\infty,2}\|\n u\|^{2(1-\tha)}_{p,q}\|\n(u_m-u)\|_{p,q}\,,
\end{split}
\end{equation*}
where we have applied H\"{o}lder inequality in the time variable with exponents $\gamma_1\,,\gamma_2$ such that
\begin{equation*}
2(1-\tha)\gamma_1=p\,,\quad \gamma_2=p\,, \quad \f{1}{\gamma_1}+\f{1}{\gamma_2}=1\,,
\end{equation*}
which gives $p=\f{5q}{5q-6}$. This shows that $\left|\int_0^{t_0}(u\cdot\n (u_m-u),u)dt\right|\to0$. Similarly, we can conclude that
$\left|\int_0^{t_0}(u\cdot\n (u_m-u),u)dt\right|\to0$.

Hence, we have
\begin{equation*}
0=\int_0^{t_0}(u\cdot\n u_m,u_m)dt\to \int_0^{t_0}(u\cdot\n u,u)dt\,.
\end{equation*}
Going through the same lines of the previous estimates, we can bound the terms $\int_0^{t_0}(u\cdot\n u,(u_m)_{\e}-(u)_{\e})dt$ and
$\int_0^{t_0}(u\cdot\n u,(u)_{\e}-u)dt$ by some finite quantity, times appropriate norms of $((u_m)_{\e}-(u)_{\e})$ and $((u)_{\e}-u)$
respectively, which converge to zero. All of this shows that
\[
\int_0^{t_0}(u\cdot\n u,\n(u_m)_{\e})dt\to0\,.
\]
The integral $\int_0^{t_0}(u,\p_t(u_m)_{\e})dt$ vanishes since $k_{\e}$ is even. By appealing to the properties of mollifiers,
\begin{equation*}
\int_0^{t_0}\left(|D(u)|^{r-2}D(u),D\left(\left(u_m\right)_{\e}\right)\right)dt\to \int_0^{t_0}\int_{\Omega}|D(u)(t)|^rdxdt\,.
\end{equation*}
Hence, in the limits, one has
\[
\|u(t_0)\|_2^2+2\int_0^{t_0}\int_{\Omega}|D(u)(t)|^rdxdt=\|u_0\|_2^2\,,
\]
for all $t_0$.
\end{proof}

\begin{proof}[Proof of Theorem \ref{thm:BV}]
In the following we always assume $p=\f{9q}{8q-9}$. We state by proving $\left|\int_0^{t_0}(u\cdot\n u_m,(u_m-u))dt\right|\to0$ and $\left|\int_0^{t_0}(u\cdot\n (u_m-u),u)dt\right|\to0$, as $m\to\infty$, where $\{u_m\}$ is a sequence in $C^{\infty}(0,\infty;C^{\infty}_0(\Omega))$ converging to $u$ in
\[
L^r(0,T;V_r)\cap L^{p_1}(0,T;L^{q_1})\cap L^{p}(0,T;W^{1,q})\,.
\]

(i) In case $\f98<q\leq\f95$ we have $u\in L^{p_1}(0,T;L^{q_1})$ with $\f{2}{p_1}+\f{2}{q_1}=\f{10}{9}$, and $q_1\geq 2q'$. We can interpolate and write $\tilde{q}$ as
\begin{equation}\label{eq:1}
\f{1}{\tilde{q}}=\f{\tha}{q^*}+\f{1-\tha}{q_1}\,.
\end{equation}
Now, we can estimate the integral as follows
\begin{equation*}
\begin{split}
\left|\int_0^{t_0}(u\cdot\n(u_m-u),u)dt\right|\leq&\left|\int_0^{t_0}\|u\|_{\tilde{q}}\|\n(u_m-u)\|_q\|u\|_{\tilde{q}}dt\right| \\
\leq&\left|\int_0^{t_0}\|u\|^{2\tha}_{q^*}\|u\|^{2(1-\tha)}_{q_1}\|\n(u_m-u)\|_qdt\right|\\
\leq&\|\n u\|^{2\tha}_{p,q}\|u\|^{2(1-\tha)}_{p_1,q_1}\|\n(u_m-u)\|_{p,q}\,,
\end{split}
\end{equation*}
where we have applied H\"{o}lder inequality in the time variable with exponents $x\,,y\,,z$ such that
\begin{equation*}
2(1-\tha)x=p_1,\quad2\tha y=p,\quad z=p\,, \quad \f{1}{x}+\f{1}{y}+\f{1}{z}=1\,,
\end{equation*}
which gives
\begin{equation}\label{eq:2}
\f{2(1-\tha)}{p_1}+\f{2\tha}{p}=1-\f{1}{p}\,.
\end{equation}
From \eqref{eq:1}, one has
\begin{equation}\label{eq:3}
\f{2(1-\tha)}{q_1}+\f{2\tha}{q^*}=1-\f{1}{q}\,.
\end{equation}
We deduce from \eqref{eq:2} and \eqref{eq:3} that
\begin{equation*}
(1-\tha)\left(\f{2}{p_1}+\f{2}{q_1}\right)=2-\f{1}{p}-\f{1}{q}-\tha\left(\f{2}{q^*}+\f{2}{p}\right)=\f{10}{q}(1-\tha)\,,
\end{equation*}
which implies $\f{2}{p_1}+\f{2}{q_1}=\f{10}{9}$. Note that, when $\tha=1$, i.e. $q=\f95$, it is not necessary to impose the condition $u\in L^{p_1}(0,T;L^{q_1})$. This is reasonable due to Theorem \ref{thm:BC}.
This shows that $\left|\int_0^{t_0}(u\cdot\n (u_m-u),u)dt\right|\to0$. Similarly, we have $\left|\int_0^{t_0}(u\cdot\n (u_m-u),u)dt\right|\to0$.

(ii) In case $q>\f95$ we have $\n u\in L^{p_1}(0,T;L^{q_1})$ with $\f{2}{p_1}+\f{2}{q_1}=\f{10}{9}$, and $\f95<q_1\leq2q'$. Here we assume $q_1>\f95$ to make the denominator of $p_1=\f{9q_1}{5q_1-9}$ be a positive value. By Gagliardo-Nirenberg inequalities, we have
\begin{equation*}
\|u\|_{\tilde{q}}\leq \|u\|^{\tha}_{q_1}\|\n u\|^{1-\tha}_q\,,
\end{equation*}
where
\begin{equation}\label{eq:21}
\f{1}{\tilde{q}}=(1-\tha)\left(\f1{q}-\f13\right)+\f{\tha}{q_1}\,.
\end{equation}
Now we can estimate, by Sobolev embeddings, the integral as follows
\begin{equation*}
\begin{split}
\left|\int_0^{t_0}(u\cdot\n(u_m-u),u)dt\right|\leq&\left|\int_0^{t_0}\|u\|_{\tilde{q}}\|\n(u_m-u)\|_q\|u\|_{\tilde{q}}dt\right| \\
\leq&\left|\int_0^{t_0}\|u\|^{2\tha}_{2}\|\n u\|^{2(1-\tha)}_{q}\|\n(u_m-u)\|_qdt\right|\\
\leq&\|u\|^{2\tha}_{p_1,q_1}\|\n u\|^{2(1-\tha)}_{p,q}\|\n(u_m-u)\|_{p,q}\,,
\end{split}
\end{equation*}
where we have applied H\"{o}lder inequality in the time variable with exponents $x\,,y\,,z$ such that
\begin{equation*}
2(1-\tha)x=p,\quad2\tha y=p_1,\quad z=p\,, \quad \f{1}{x}+\f{1}{y}+\f{1}{z}=1\,,
\end{equation*}
which gives
\begin{equation}\label{eq:22}
\f{2(1-\tha)}{p}+\f{2\tha}{p_1}=1-\f{1}{p}\,.
\end{equation}
From \eqref{eq:21}, one has
\begin{equation}\label{eq:23}
2(1-\tha)\left(\f{1}{q}-\f13\right)+\f{2\tha}{q_1}=1-\f{1}{q}\,.
\end{equation}
We deduce from \eqref{eq:22} and \eqref{eq:23} that
\begin{equation}\label{}
\tha\left(\f{2}{p_1}+\f{2}{q_1}\right)=2-\f{1}{p}-\f{1}{q}-(1-\tha)\left(\f{2}{p}+\f{2}{q}-\f23\right)=\f{10}{q}\tha\,,
\end{equation}
which gives $\f{2}{p_1}+\f{2}{q_1}=\f{10}{9}$.
This shows that $\left|\int_0^{t_0}(u\cdot\n (u_m-u),u)dt\right|\to0$. Similarly, we have $\left|\int_0^{t_0}(u\cdot\n (u_m-u),u)dt\right|\to0$.

Hence, we have
\begin{equation*}
0=\int_0^{t_0}(u\cdot\n u_m,u_m)dt\to \int_0^{t_0}(u\cdot\n u,u)dt\,.
\end{equation*}

The remaining procedure is same as the proof of Theorem \ref{thm:BC}.
\end{proof}

\end{document}